\documentclass{amsart} 

\title{Bounded distortion homeomorphisms on ultrametric spaces}

\author[B.~Hughes]{Bruce Hughes}
\address{Department of Mathematics\\ Vanderbilt University\\ Nashville, TN 37240 U.S.A.}
\email{bruce.hughes@vanderbilt.edu}
\thanks{The first-named author is supported in part by NSF Grant DMS--0504176.}

\author[Á.~Martínez-Pérez]{Álvaro Martínez-Pérez}
\address{Departamento de Geometría y Topología\\ Universidad Complutense de Madrid\\ Madrid 28040  Spain}
\email{@mat.ucm.es}
       
\author[M.~A.~Morón]{Manuel A. Morón}
\address{Departamento de Geometría y Topología\\ Universidad Complutense de Madrid\\ Madrid 28040  Spain}
\email{mamoron@mat.ucm.es}
\thanks{The second- and third-named authors are partially supported by MTM 2006-00825.}

\dedicatory{Dedicated to José María Montesinos on the occasion of his 65th birthday}

\usepackage[latin1]{inputenc}
\usepackage{amsfonts}
\usepackage{amsmath}
\usepackage{latexsym}
\usepackage[all]{xy}
\usepackage{amsthm}
\usepackage{graphicx}
\usepackage[T1]{fontenc}
\usepackage{hyperref}

\usepackage{caption}
\usepackage{verbatim}

\usepackage{amscd}
\usepackage{pinlabel}
\usepackage{graphics}
\usepackage{graphpap}
\usepackage{pb-diagram}

\newtheorem{definicion}{Definition}[section]

\newtheorem{prop}[definicion]{Proposition}

\newtheorem{teorema}[definicion]{Theorem}

\newtheorem{ejp}[definicion]{Example}

\newtheorem{corollary}[definicion]{Corollary} 

\theoremstyle{definition} 
\newtheorem{definition}[definicion]{Definition} 

\newtheorem{notation}[definicion]{Notation}

\theoremstyle{remark}
\newtheorem{remark}[definicion]{Remark}

\newcommand{\br}{\ensuremath{\mathbb{R}}} 
\newcommand{\co}{\ensuremath{\colon}} 
\newcommand{\bn}{\ensuremath{\mathbb{N}}} 
\newcommand{\e}{\ensuremath{\mathrm{e}}} 

\begin{document}

\begin{abstract} 
It is well-known that quasi-isometries between $\br$-trees induce power quasi-symmetric homeomorphisms
between their ultrametric end spaces.
This paper investigates power quasi-symmetric homeomorphisms between bounded, complete, uniformly perfect, ultrametric spaces (i.e., those
ultrametric spaces arising up to similarity as the end spaces of bushy trees).
A bounded distortion property is found that characterizes  power quasi-symmetric homeomorphisms
between such ultrametric spaces that are also pseudo-doubling. 
Moreover, examples are given showing the extent to which the power quasi-symmetry of homeomorphisms is not captured by the
quasiconformal and bi-H\"older conditions for this class of ultrametric spaces.
\end{abstract}

\maketitle
\tableofcontents

\begin{footnotesize}
\end{footnotesize}

\section{Introduction}
\label{sec:introduction}

A theme in the study of noncompact spaces is the extent to which the geometry and topology of 
such a space is reflected in a natural boundary at infinity.

By choosing a root $v$ for an $\br$-tree $T$, the boundary at infinity naturally becomes a complete, ultrametric space $end(T,v)$ of diameter $\leq 1$,
called the \emph{end space} of $(T,v)$.
It is known that a quasi-isometry between $\br$-trees induces a bi-H\"older, quasiconformal homeomorphism 
between their ultrametric end spaces. In fact, the induced homeomorphism has the stronger power quasi-symmetric, or PQ-symmetric, property.

This paper is concerned with the natural question, How close do the bi-H\"older and quasiconformal  conditions come to characterizing PQ-symmetric homeomorphisms on bounded, complete ultrametric spaces? 
What if one restricts to bounded, complete, uniformly perfect, ultrametric spaces? 
The significance of restricting to this class of ultrametric spaces is that they are (up to similarity) exactly the ones that arise as end spaces of
bushy $\br$-trees. 
We introduce a notion of bounded distortion for homeomorphisms and 
a pseudo-doubling property of metric spaces.
We show that the bounded distortion property characterizes PQ-symmetric homeomorphisms on
bounded, complete, uniformly perfect, ultrametric spaces that also have the pseudo-doubling property. 
The following is a statement of our main positive result.

\begin{teorema}\label{thm:main} 
A homeomorphism $h\co X\to Y$ between bounded, complete, uniformly perfect, pseudo-doubling, ultrametric spaces
is PQ-symmetric if and only if $h$ is a bounded distortion equivalence. 
\end{teorema}

The class of ultrametric spaces to which Theorem~\ref{thm:main} applies includes
end spaces of rooted, geodesically complete, bushy, simplicial $\br$-trees---see Remark~\ref{remark:pseudo-doubling}.

The more difficult of the proof of Theorem~\ref{thm:main} is in showing that bounded distortion equivalences on the given
class of ultrametric spaces are PQ-symmetric. This is accomplished in Corollary~\ref{cor:BDE PQ for pd}. The proof of that
corollary  relies on
Theorem~\ref{prop:BQC implies PQ}, which establishes Theorem\ref{thm:main} for the case of end spaces of rooted, geodesically complete, 
simplicial, bushy $\br$-trees. 
The converse of Theorem~\ref{thm:main} is Proposition~\ref{prop:QS implies BQE}. 

In Example~\ref{example:BDE not PQS}
it is shown that this theorem does not hold for compact, uniformly perfect, ultrametric spaces that are not pseudo-doubling.
Moreover, we give several examples in Section~\ref{sec:examples} illuminating the difference between 
bi-H\"older, quasiconformal homeomorphisms on one hand,  and PQ-symmetric homeomorphisms on the other hand for this class of ultrametric spaces.
In fact, the examples are defined on end spaces of locally finite, simplicial trees of minimal vertex degree three and answer
several questions of Mirani \cite{Mirani2006, Mirani2008}.

The study of  quasi-isometries between trees and 
the induced maps on their end spaces has a  voluminous literature.
This is often set in the more general context of hyperbolic metric spaces
and their boundaries.
See Bonk and Schramm \cite{BoSchr},
Buyalo and Schroeder \cite{BS},
Ghys and de la Harpe \cite{GhysdelaHarpe},
Mart\'inez-P\'erez \cite{M},
Mirani \cite{Mirani2006, Mirani2008},
and Paulin \cite{Pau}
to name a few. 
For homeomorphisms induced by $\br$-tree morphisms that are less restrictive than quasi-isometries, see Mart\'inez-P\'erez and Mor\'on \cite{M-M}.
For $\br$-tree morphisms more restrictive 
than quasi-isometries, see Hughes \cite{Hug}.


\section{Preliminaries on trees, end spaces, and ultrametrics}

In this section,  we recall the definitions of the trees and their end spaces that are
relevant to this paper. 
We also describe a well-known correspondence between trees and ultrametric
spaces. See Fe{\u\i}nberg \cite{Feinberg} for an early result along these lines and  Hughes \cite{Hug} for additional background.

\begin{definition} Let $(T,d)$ be a metric space. 
\begin{enumerate}
	\item $(T,d)$ is an \emph{$\mathbb{R}$--tree} if $T$ is uniquely arcwise connected and for all $x, y \in T$, the
unique arc from $x$ to $y$, denoted $[x,y]$, is isometric to the
subinterval $[0,d(x,y)]$ of $\mathbb{R}$.
\item
A \emph{rooted $\mathbb{R}$--tree} $(T,v)$ consists of an $\br$-tree $(T,d)$ and a point $v\in T$, called  the \emph{root}.
\item
\label{extensiongeod} A rooted $\mathbb{R}$--tree $(T,v)$ is
\emph{geodesically complete} if every isometric embedding
$f\co [0,t]\rightarrow T$ with $t>0$ and  $f(0)=v$ extends to an isometric embedding $F\co [0,\infty) \rightarrow T$. 
\item A \emph{simplicial $\mathbb{R}$--tree} is an $\br$-tree $(T,d)$ such that $T$ is the (geometric realization of) a simplicial complex
and every edge of $T$ is isometric to the closed unit interval $[0,1]$.
\end{enumerate}
\end{definition}

\begin{definition} An \emph{ultrametric space} is a metric space $(X,d)$ such that 
$d(x,y)\leq \max \{d(x,z),d(z,y)\}$
for all $x,y,z\in X$. 
\end{definition}

\begin{definition}\label{end} The \emph{end space} of a rooted
$\mathbb{R}$--tree $(T,v)$ is given by: 
$$end(T,v)=\{F\co [0,\infty) \rightarrow T \ |\ \text{$F(0)=v$ and $F$
is an isometric embedding}\}.$$ 
Let $F, G\in end(T,v)$.
\begin{enumerate}
	\item  The {\it Gromov product at infinity} is $(F|G)_v :=\sup \{t\geq 0 \ |\ F(t)=G(t)\}$.
	\item  The {\it end space metric} is $d_v(F,G) := \e^{-(F|G)_v}$.
	\item The arc $F([0, (F|G)_v])$ is denoted $[F|G]$.
	\item The \emph{bifurcation point} of $F$ and $G$ is $F((F|G)_v)\in T$.
\end{enumerate}
\end{definition}

\begin{prop} If $(T,v)$ is a rooted $\mathbb{R}$--tree,
then $(end(T,v),d_v)$ is a complete ultrametric space of diameter
$\leq 1$.\qed
\end{prop}

\begin{definition}\label{tu}
Let $(U, d)$ be a complete ultrametric space with diameter $\leq 1$. Define an equivalence relation
$\sim$ on $U\times [0,\infty)$ by:
$$(x,t)\sim(y,t')\Longleftrightarrow t=t' \text{ and }  d(x,y)\leq \e^{-t}.$$
Then $T_U := U\times [0,\infty)/\sim$ is the \emph{tree associated to $(U,d)$}.
Define a metric $D$ on $T_U$  by:
\[D([x,t],[y,s])=\left\{
\begin{tabular}{l} $|t-s| \qquad \qquad \qquad \qquad \qquad
\qquad  \mbox{ if }
x=y,$\\

$t+s-2\min\{-\ln(d(x,y)),t,s\} \quad \mbox{ if } x\ne
y.$\end{tabular}
 \right.\]
\end{definition}

A proof of the first item in the following proposition can be
found in \cite[Theorem 6.3]{Hug} and a proof of the second item can be found in \cite[Proposition 6.4]{Hug} and \cite[Proposition 5.5]{M-M}. 

\begin{prop}
\label{prop:assoc tree}
 Let $(U,d)$ be a complete, ultrametric space of diameter $\leq 1$ and let $(T_U, D)$ be its associated tree.
\begin{enumerate}
	\item  
$(T_U,D)$ is a rooted, geodesically complete $\mathbb{R}$--tree with root $v= [x,0]$ for any $x\in U$.
\item\label{isometry} $U$ is isometric to $end(T_U, v)$. \qed
\end{enumerate}
\end{prop}

In this article a \emph{map} is a function that need not be continuous.

\begin{definition}
\label{def:quasi-isometry}
A map $f\co X\to Y$ between metric spaces $(X,d_X)$ and $(Y,d_Y)$ is
a \emph{quasi-isometric map} if there are constants $\lambda
\geq 1$ and $A>0$ such that for all $x,x'\in X$,
$$\frac{1}{\lambda}d_X(x,x') -A \leq d_Y(f(x),f(x'))\leq \lambda
d_X(x,x')+A.$$ If $f(X)$ is a net in $Y$ (i.e., there exists $\epsilon >0$ such that for each $y\in Y$ there exists $x\in X$ such that
$d_Y(f(x), y) <\epsilon$), then $f$ is a
\emph{quasi-isometry} and $X,Y$ are \emph{quasi-isometric}.
\end{definition}

\begin{remark} It is well-known that a quasi-isometry $f\co T\to T'$ between rooted $\br$-trees $(T,v)$ and $(T',w)$ induces a homeomorphism
$\tilde{f}\co end(T,v)\to end(T',w)$. See, for example,
Bridson and Haefliger \cite[Chapter I.8]{BridsonHaefliger}, where they work in the more general setting of
proper, geodesic, metric spaces.

For a quasi-isometry $f\co X\to Y$ between Gromov hyperbolic, almost geodesic metric spaces, Bonk and Schramm 
define \cite[Proposition 6.3]{BoSchr} the induced map $\partial f\co\partial X\to\partial Y$ between the boundaries at infinity
and prove \cite[Theorem 6.5]{BoSchr} that $\partial f$ is  PQ-symmetric (see Definition~\ref{def:PQ-symmetric}  below) 
with respect to any metrics on $\partial X$ and $\partial Y$ 
in their canonical gauges. 
In the special case that $X$ and $Y$ are $\br$-trees, $\partial X= end(X,v)$, $\partial Y=end(Y,w)$ and the end space metrics
are in the canonical gauges for any choice of roots.

Another source for the result that quasi-isometries between $\br$-trees induce PQ-symmetric homeomorphisms on their ultrametric end spaces,
is Buyalo and Schroeder \cite[Theorem 5.2.17]{BS}. They work with Gromov hyperbolic, geodesic metric spaces and with visual boundaries on their
boundaries. When specialized to $\br$-trees, these boundaries are the ultrametric end spaces. 
\end{remark}

\section{Homeomorphisms on metric spaces}
\label{sec:homeo metric}

In this section we begin our discussion of various geometric properties that may be satisfied by homeomorphisms 
between metric spaces. Homeomorphisms between end spaces of rooted, geodesically complete $\br$-trees induced by quasi-isometries of the trees 
are examples where these properties are encountered. With the possible exceptions of homeomorphisms of bounded distortion
and bounded distortion equivalences, defined in 
Definition~\ref{def:qc et al} below, all 
of these concepts are well-known. In addition to Buyalo and Schroeder \cite{BS}, other sources for background include
Bonk and Schramm \cite{BoSchr},
Heinonen \cite{Heinonen},  Roe \cite{Roe1}, Semmes \cite{Semmes}, and 
Tukia and V{\"a}is{\"a}l{\"a} \cite{TukVais}.

\begin{definition}
Let $f\co X\to Y$ be a homeomorphism between metric spaces
$(X,d_X)$ and $(Y,d_Y)$. If 
$x_0\in X$ and $\epsilon >0$, then the {\it distortion by $f$ of
the $\epsilon$-sphere $S(x_0,\epsilon) := \{ x\in X ~|~ d_X(x_0,x) =\epsilon\}$ at $x_0$} is
$$D_f(x_0,\epsilon) := 
\begin{cases}
\frac{\sup\{ d_Y(f(x_0),f(x)) ~|~ d_X(x_0,x)=\epsilon\}}{\inf\{d_Y(f(x_0),f(x)) ~|~ d_X(x_0,x)=\epsilon\}} & \text{if $S(x_0,\epsilon)\not=\emptyset$}\\
\quad\quad\quad\quad\quad\quad 1 & \text{if $S(x_0,\epsilon)=\emptyset$}
\end{cases}$$
\end{definition}

\begin{definition}
\label{def:qc et al}
Let $f\co X\to Y$ be a homeomorphism between  metric spaces. 
\begin{enumerate}
	\item $f$ is {\it conformal} if $\underset{\epsilon\to 0}\limsup D_f(x_0,\epsilon) =1$ for all $x_0\in X$.
		\item $f$ is {\it $K$-quasiconformal}, where $K>0$, if $\underset{\epsilon\to 0}\limsup D_f(x,\epsilon) \leq K$ for all $x\in X$.
			\item $f$ is {\it quasiconformal} if $f$ is $K$-quasiconformal for some $K>0$.
			\item $f$ has {\it bounded distortion} if there exists $K>0$ such that
$$\underset{x\in X}\sup\, \underset{\epsilon> 0}\sup\, D_f(x,\epsilon) \leq K.$$
\item $f$ is a {\it bounded distortion equivalence} if $f$ and $f^{-1}\co Y\to X$ have bounded distortion.
\end{enumerate}
\end{definition}

Quasiconformal maps are those maps with control on the distortion of sufficiently small spheres. 
Homeomorphisms of bounded distortion have control on the distortion of every sphere.
An example of a quasiconformal (in fact, conformal) homeomorphism that is not of bounded distortion is provided by Example~\ref{quasi-conf}
below.

\begin{definition}
\label{def:quasi-symmetric} A map $f\co X \to Y$ between metric spaces 
$(X,d_X)$ and $(Y,d_Y)$ is
\emph{quasi-symmetric} if $f$ is not constant and if there
is a homeomorphism $\eta\co [0,\infty) \to [0,\infty)$, called the {\it control function for $f$}, such that whenever
$x, a, b\in X$, $t\geq 0$, and
$d_X(x,a)\leq t\,d_X(x,b)$ it follows that $d_Y(f(x),f(a))\leq
\eta(t)\,d_Y(f(x),f(b))$. 
\end{definition}

Thus, a quasi-symmetric map controls distortion of annuli; more precisely, it controls the distortion of ratios between inner and outer radii
of annuli. 

\begin{definition}
\label{def:PQ-symmetric} 
A quasi-symmetric map is said to be \emph{power quasi-symmetric},
or \emph{PQ-symmetric}, if it has a control function  of the form
\[\eta(t)= q \max\{t^p, t^{1/p} \}\] for some $p,q\geq 1$.
\end{definition}

\begin{remark}
\label{rem:TV inverses}
It follows from Tukia and V{\"a}is{\"a}l{\"a} \cite[Theorem 2.2]{TukVais} that inverses and compositions of quasi-symmetric homeomorphisms
are quasi-symmetric; moreover, inverses and compositions of PQ-symmetric homeomorphisms
are PQ-symmetric.
\end{remark}

\begin{prop}
\label{prop:QS implies BQE} 
If a homeomorphism $f\co X \to Y$ between metric spaces
$(X,d_X)$ and $(Y,d_Y)$
is quasi-symmetric, then $f$ 
is a bounded distortion equivalence.
\end{prop}

\begin{proof} 
Let $\eta$ be a control function for $f$ and let $x\in X$ and $\epsilon >0$ be given. 
If $a,b \in X$ with $d_X(x,a)= \epsilon =d_X(x,b)$, then
$d_Y(f(x),f(a))\leq \eta(1)\, d_Y(f(x),f(b))$.
Hence, $D_f(x,\epsilon)\leq \eta(1)$ and $f$ has bounded distortion with constant $K=\eta(1)$.
It follows from Remark~\ref{rem:TV inverses} that $f^{-1}$ is also quasi-symmetric (hence, of bounded distortion)
and $f$ is a bounded distortion equivalence.
\end{proof}

\begin{definition}
\label{def:bi-holder}
A homeomorphism $f\co X \to Y$ between metric spaces
$(X,d_X)$ and $(Y,d_Y)$
is
{\it bi-H\"older} if there exists constants $\alpha > 0$ and $c>0$ such that 
$$\frac{1}{c} d_X(x,y)^{1/\alpha} \leq d_Y(f(x),f(y)) \leq c d_X(x,y)^\alpha$$
for all $x, y\in X$. If $\alpha=1$, then $f$ is {\it bi-Lipschitz}.
\end{definition}

\begin{remark}
\label{remark:PQ biH}
Note that a bi-Lipschitz homeomorphism is PQ-symmetric with $p=1$, $q=c^2$, and $\eta(t) = qt$.
It follows from the proof of Tukia and V{\"a}is{\"a}l{\"a} \cite[Theorem 3.14]{TukVais}
that a PQ-symmetric homeomorphism $f\co X\to Y$ between bounded metric spaces is bi-H\"older.
A quasiconformal homeomorphism need not be bi-H\"older as shown in Example~\ref{ex:qc not biholder}
below.
\end{remark} 

\begin{remark}
For homeomorphisms between arbitrary bounded metric spaces, there are the following implications:
\begin{small}$$
\begin{diagram}
\node{\text{PQ-symmetric}} \arrow{s} \arrow{e} \node{\text{quasi-symmetric}} \arrow{e} \node{\text{bounded distortion equivalence}}\arrow{s} \\
\node{\text{bi-H\"older}} \arrow{e,!} \node{\text{conformal}} \arrow{e} \node{\text{quasiconformal}}
\end{diagram}
$$\end{small}
In fact, the implication that PQ-symmetric homeomorphisms are bi-H\"older is the only one that requires that the metric spaces be bounded.
No other implications hold in general.
\end{remark} 
\section{Homeomorphisms on ultrametric spaces}
\label{sec:homeo ultrametric}

In this section, the Gromov product at infinity is used
to characterize homeomorphisms  of bounded distortion, quasi-symmetries,  and PQ-symmetries
in the case of interest to us, namely,
maps between end spaces of trees. 
It is also shown that local similarity equivalences between compact, ultrametric spaces are PQ-symmetric.

We begin with the characterization of bounded distortion homeomorphisms in terms of the Gromov product at infinity.

\begin{prop}
\label{prop:bqc gpi} If $(T,v)$ and $(T',w)$ are  rooted, geodesically complete
$\mathbb{R}$--trees, then a homeomorphism $h\co  end(T,v) \to end(T',w)$ has bounded distortion
if and only if there exists a constant $A\geq 0$ such
that whenever $F,G,H \in end(T,v)$ and
$(F|G)_v=(F|H)_v$ it follows that $|(h(F)|h(G))_w-(h(F)|h(H))_w|\leq A$.
\end{prop}

\begin{proof} 
Using the relation between the end space metric and the Gromov product at infinity, it follows that 
for $F\in end(T,v)$, $\epsilon >0$, and $K\geq 1$, we have
$D_h(F,\epsilon)\leq K$ if and only if
$$\frac{\e^{-(h(F)|h(G))_w}}{\e^{-(h(F)|h(H))_w}} \leq K,$$
whenever
$G, H \in end(T,v)$ and $(F|G)_v =(F|H)_v = -\ln\epsilon$.
Thus, if 
$F\in end(T,v)$ and $K\geq 1$, then
$\sup_{\epsilon >0} D_h(F,\epsilon)\leq K$ 
if and only if
$$\e^{(h(F)|h(H))_w -(h(F)|h(G))_w} \leq K,$$
whenever 
$G, H \in end(T,v)$ and $(F|G)_v= (F|H)_v $;
in turn, this holds
if and only if 
$$\vert(h(F)|h(H))_w -(h(F)|h(G))_w\vert \leq\ln K,$$
whenever 
$G, H \in end(T,v)$ and $(F|G)_v= (F|H)_v $.
The result follows; furthermore, the relationship between the bounded distortion
constant $K$ and the constant $A$ is given by $A=\ln K$.
\end{proof}

The following result gives the characterization of quasi-symmetric maps between end spaces of trees in terms
of the Gromov product at infinity. 

\begin{prop}
\label{prop:qs gpi} If $(T,v)$ and $(T',w)$ are  rooted, geodesically complete
$\mathbb{R}$--trees, then a map $f\co  end(T,v) \to end(T',w)$
is quasi-symmetric
if and only if there exists an orientation-preserving homeomorphism $\gamma\co\br\to \br$ such
that whenever $F,G,H\in end(T,v)$,
it follows that  
$$-\gamma((F|G)_v-(F|H)_v)  \leq (F'|H')_w-(F'|G')_w\leq
\gamma((F|H)_v-(F|G)_v),$$
where $F'=f(F),G'=f(G)$, and $H'=f(H)$.
\end{prop}

\begin{proof} 
According to Tukia and V{\"a}is{\"a}l{\"a} \cite[page 99]{TukVais}, $f$ is quasi-symmetric if and only if there exists a homeomorphism
$\eta\co [0,\infty)\to[0,\infty)$ such that
$$\eta(\rho^{-1})^{-1}\leq \frac{d_w(F',G')}{d_w(F',H')} \leq \eta(\rho)$$
whenever $F,G,H\in end(T,v)$ and $\rho =\frac{d_v(F,G)}{d_v(F,H)}$.
The result now follows easily via the correspondence of $\eta$ and $\gamma$ implied by the relation
$\gamma(t) =\ln\eta(\e^t)$ and
the relation between the end space metric and the Gromov product at infinity.
\end{proof}

The following result gives the characterization of PQ-symmetric maps between end spaces of trees in terms
of the Gromov product at infinity. 

\begin{prop}
\label{prop:pqs gpi} If $(T,v)$ and $(T',w)$ are  rooted, geodesically complete
$\mathbb{R}$--trees, then a map $f\co  end(T,v) \to end(T',w)$
is PQ-symmetric
if and only if there exist constants $\lambda\geq 1$ and $A\geq 0$ such
that whenever $F,G,H\in end(T,v)$ and $(F|G)_v\leq (F|H)_v< \infty$,
it follows that  
\[\frac{1}{\lambda} ((F|H)_v-(F|G)_v)-A \leq (F'|H')_w-(F'|G')_w\leq
\lambda ((F|H)_v-(F|G)_v)+A,\]
where $F'=f(F),G'=f(G)$ and $H'=f(H)$.
\end{prop}

\begin{proof} 
Suppose first that $f$ is PQ-symmetric with constants $p,q\geq 1$
as in Definition~\ref{def:PQ-symmetric} and let
$\lambda=p$ and $A=\ln q$.
If $F,G,H\in end(T,v)$ and $(F|G)_v\leq (F|H)_v< \infty$, then it follows from
the relation between the end space metric and the Gromov product at infinity  that
$$\frac{d_v(F,G)}{d_v(F,H)} = \e^{(F|H)_v-(F|G)_v}\geq 1.$$
Thus,  $d_v(F,G)=t\cdot d_v(F,H)$, where $t\geq 1$ and
$\ln t = (F|H)_v-(F|G)_v$. 
The PQ-symmetric property implies  
$$d_w(F',G')\leq qt^p d_w(F',H')$$ and
$$d_w(F',H')\leq q(1/t)^{1/p}d_w(F',G').$$ 
These two inequalities are equivalent to  
\[(F'|G')_w\geq
(F'|H')_w-\lambda ((F|H)_v-(F|G)_v)-A\] and
\[(F'|H')_w\geq (F'|G')_w+\frac{1}{\lambda} ((F|H)_v-(F|G)_v)-A,\] 
which are in turn equivalent to the desired inequalities.

Conversely, given $\lambda\geq 1$ and $A\geq 0$ satisfying the given conditions,
let $p=\lambda$ and $q=\e^A$.
Suppose $F, G, H\in end(T,v)$, $t\geq 0$, and $d_v(F,G)\leq t\, d_v(F,H)$. It must be shown that
$$d_w(F', G')\leq\begin{cases} 
qt^p\, d_w(F', H') & \text{if $t\geq 1$}\\
qt^{1/p}\, d_w(F', H') & \text{if $t\leq 1$}
\end{cases}$$
(where $F', G', H'$ continue to denote the images of $F, G, H$, respectively, under $f$).
We may assume $0 < d_v(F,G)$ and $t>0$, for otherwise the result is trivial.

We first consider the case $(F|G)_v\leq (F|H)_v$. In particular, $d_v(F,H)\leq d_v(F,G)$ and $t\geq 1$.
The right-hand part of the assumed inequalities takes the form
$$\ln\frac{d_w(F',G')}{d_w(F',H')} \leq p\ln\frac{d_v(F,G)}{d_v(F,H)}+\ln q =\ln\left[q\left(\frac{d_v(F,G)}{d_v(F,H)}\right)^p\right].$$
Therefore, 
$$\frac{d_w(F',G')}{d_w(F',H')} \leq q\left(\frac{d_v(F,G)}{d_v(F,H)}\right)^p \leq qt^p,$$
as required.

Finally, consider the case $(F|H)_v\leq (F|G)_v$.
The left-hand part of the assumed inequalities takes the form
$$\ln\left[\frac{1}{q}\left(\frac{d_v(F,H)}{d_v(F,G)}\right)^{1/p}\right]   =\frac{1}{p}\ln\frac{d_v(F,H)}{d_v(F,G)}-\ln q \leq \ln\frac{d_w(F',H')}{d_w(F',G')}.$$
Therefore,
$$\frac{1}{q}\left(\frac{1}{t}\right)^{1/p} \leq   \frac{1}{q}\left(\frac{d_v(F,H)}{d_v(F,G)}\right)^{1/p} \leq \frac{d_w(F',H')}{d_w(F',G')}$$
and 
$$d_w(F',G')\leq q t^{1/p}\, d_w(F',H').$$
This is exactly what is required if $t\leq 1$;
if $t\geq 1$, what is required follows by using $t^{1/p}\leq t^p$.
\end{proof}

\begin{remark}
\label{remark:gpi biH}
The following characterization of bi-H\"olderness in terms of the Gromov product at infinity is straightforward to verify.
If $(T,v)$ and $(T',w)$ are  rooted, geodesically complete
$\mathbb{R}$--trees, then a homeomorphism $f\co  end(T,v) \to end(T',w)$
is bi-H\"older 
if and only if there exist constants $\lambda\geq 1$ and $A\geq 0$ such
that whenever $F,G\in end(T,v)$,
it follows that  
$$\frac{1}{\lambda} (F|G)_v-A \leq (F'|G')_w\leq
\lambda (F|G)_v+A,$$
where $F'=f(F)$ and $G'=f(G)$.
\end{remark}

\begin{remark}
\label{remark:rooted homeo}
Suppose $(T,v)$ and $(T',w)$ are rooted, geodesically complete, $\br$-trees and 
$h\co end(T,v)\to end(T',w)$ is a homeomorphism induced by a rooted homeomorphism
$\hat{h}\co (T,v)\to (T',w)$; i.e., $h(F)=\hat{h}\circ F$ for all $F\in end(T,v)$.
It follows that if $F, G, H\in end(T,v)$ and $(F|G)_v=(F|H)_v$, then
$(h(F)|h(G))_w=(h(F)|h(H))_w$.
Thus, $h$ is conformal and a bounded distortion equivalence (with constant $K=1$).
\end{remark}

Finally, we show that the local similarity equivalences between compact ultrametric spaces studied in \cite{Hug, HugTUNG} are
PQ-symmetric. In fact, we show that they are bi-Lipschitz. In particular, this affirms a conjecture of Mirani \cite{Mirani2006, Mirani2008}.

\begin{definition} A function $f\co X \to Y$ between metric spaces
$(X,d_X)$, $(Y,d_Y)$ is a \emph{similarity} if there exists
$\lambda>0$ such that $d_Y(f(x),f(y))=\lambda d_X(x,y)$ for all
$x,y \in X$. In this case,  $f$ is a \emph{$\lambda$-similarity}. 
\end{definition}

\begin{definition} A homeomorphism $h\co X \to Y$ between metric spaces is a
\emph{local similarity equivalence} if for every $x\in X$ there
exist $\varepsilon>0$ and $\lambda>0$ such that the restriction
$h|_{B(x,\varepsilon)}\co  B(x,\varepsilon) \to B(h(x),\lambda
\varepsilon)$ is a surjective $\lambda$-similarity.
\end{definition}

\begin{prop}
\label{thm:LSE} If $f : U \to V$ is a local similarity equivalence between
compact, ultrametric spaces, then $f$ is bi-Lipschitz.
In particular, $f$ is PQ-symmetric.
\end{prop}

\begin{proof}
Up to similarity homeomorphism, we may assume that the diameters of $U$ and $V$ are $\leq 1$.
Proposition~\ref{prop:assoc tree} then implies that $U= end(T,v)$ and $V= end(T',w)$, where
$(T,v)$ and $(T',w)$ are rooted, geodesically
complete $\mathbb{R}$--trees.
There is a local similarity equivalence $h:end(T,v)\to end(T',w)$
induced by conjugating $f$ by similarities. Clearly, it suffices to 
show that $h$ is bi-Lipschitz. To this end, we will show that 
there exists some constant $K>0$ such that if $F,G \in
end(T,v)$, then $|(F|G)_v-(h(F)|h(G))_w|\leq K$ .
If there is such a constant $K$, then it is clear that: if $F,G
\in end(T,v)$, then $ (F|G)_v-K \leq
(h(F)|h(G))_w \leq(F|G)_v+K$ 
and
$e^{-K}d(F,G) \leq d(h(F),h(G)) \leq e^K d(F,G)$. That is, $h$
is bi-Lipschitz. 

To complete the proof, we will establish the existence of $K$.
For every $F\in end(T,v)$ there exist $\varepsilon>0$ and
$\lambda>0$ such that the restriction $h|_{B(x,\varepsilon)}:
B(x,\varepsilon) \to B(h(x),\lambda \varepsilon)$ is a surjective
$\lambda$-similarity. Since $end(T,v)$ is compact, there is a
finite, open covering $\{B(x_i,\varepsilon_i)\}_{i=1}^n$ and associated similarity constants
$\{\lambda_i\}_{i=1}^n$. Since $h$ is a homeomorphism, $\{B(h(x_i),\lambda_i \varepsilon_i))\}_{i=1}^n$ is also an
open covering of $end(T',w)$.
Let $\delta_1$ a Lebesgue number for the covering
$\{B(x_i,\varepsilon_i)\}_{i=1}^n$ and $\delta_2$ a Lebesgue number for the
covering $\{B(h(x_i),\lambda_ i \varepsilon_i)\}_{i=1}^n$. Define 
$$K:=
\max\{\max\{|\ln (\lambda_i)| ~|~ i=1,\dots,n \}, -2\ln(\delta_1),-2\ln(\delta_2)\}.$$
If  $F,G\in B(x_i,\varepsilon_i)$ for some $i=1,\dots, n$, then
$d(h(F),h(G))=\lambda_i d(F,G)$ and $|(F|G)_v-(h(F)|h(G))_w|=
|\ln(\lambda_i)| \leq K$, as desired.
On the other hand, if $F,G$ are not in any such ball, then 
$h(F)$ and $h(G)$ are not in the same ball
$B(h(x_i),\lambda_i \varepsilon_i))$ for any $i$; therefore,
$d(h(F),h(G))>\delta_2$. Thus, $|(F|G)_v-(h(F)|h(G))_w|\leq
(F|G)_v+(h(F)|h(G))_w\leq -\ln(\delta_1)- \ln(\delta_2)\leq K$, which
establishes the desired property of $K$.
\end{proof}

\begin{remark} An alternative proof that a local similarity equivalence between compact, ultrametric spaces
is PQ-symmetric can be obtained as follows. First, it is possible to modify the
proof of Tukia and V\"ais\"al\"a \cite[Theorem 2.23]{TukVais} to show that 
a local PQ-symmetric embedding $f\co X\to Y$ between metric spaces, where $X$ is compact, is PQ-symmetric.
Then observe that, since a similarity is PQ-symmetric (with $p=1=q$), a local similarity equivalence is locally PQ-symmetric. 
\end{remark}


\section{Homeomorphisms on uniformly perfect, ultrametric spaces}
\label{sec:homeo uniformly perfect, ultrametric}

In this section, it is shown that a bounded distortion homeomorphism  between bounded, complete, uniformly perfect,
pseudo-doubling, ultrametric spaces
is PQ-symmetric---see Corollary~\ref{cor:BDE PQ for pd}.
This is  the remaining part of the main result Theorem~\ref{thm:main}, namely, 
the sufficiency of bounded distortion for PQ-symmetry.
The proof of Corollary~\ref{cor:BDE PQ for pd} consists of a reduction to the end spaces of
simplicial, bushy  $\br$-trees. That reduction is contained in Theorem~\ref{prop:BQC implies PQ}. 
This section also contains the definition of pseudo-doubling metric spaces and a brief discussion of its
relationship to the well-known notion of doubling metric spaces. 

We begin by recalling the following definition of uniformly perfect for metric spaces and its relation to
bushy $\br$-trees.

\begin{definition} A metric space $X$ is {\it uniformly perfect} if
there is a constant $\mu \in (0,1)$ such that for every $x\in X$ and
every $r>0$, it follows that $B_r(x)\backslash B_{\mu r}(x)\neq \emptyset$
unless $X= B_r(x)$.
\end{definition}

Tukia and V{\"a}is{\"a}l{\"a} \cite{TukVais}
proved that a quasi-symmetric homeomorphism between uniformly perfect metric spaces is PQ-symmetric (see also
\cite[Theorem 11.3, page 89]{Heinonen}).

Now recall the following definition from Mosher, Sageev, and Whyte \cite{Mo}.

\begin{definition} An $\br$-tree $T$ is \emph{bushy} if there is
a constant $K>0$, called a {\it bushy constant}, such that for any point $x\in T$ there is
a point $y\in T$ such that $d(x,y)< K$ and $T\backslash \{y\}$ has at least
$3$ unbounded components.
\end{definition}

\begin{remark}\label{bushy-unif}
Note that a rooted, geodesically complete $\br$-tree $T$ is bushy if and only if 
$end(T,v)$ is uniformly perfect for some (respectively, for every) $v\in T$.
\end{remark}

\begin{remark}\label{unique tree} As Mosher, Sageev, and Whyte \cite[page 118]{Mo} point out,
any two bounded valence, locally finite, simplicial, bushy trees
are quasi-isometric. 
Therefore, any such tree is quasi-isometric to the infinite binary tree.
See also Bridson and Haefliger \cite[page 141, Exercise 8.20(2)]{BridsonHaefliger} for the special case of
regular, simplicial trees. 
\end{remark}

\begin{teorema}
\label{prop:BQC implies PQ} 
If a map $h\co end(T,v)\to end(T',w)$ between the end spaces
of rooted, geodesically complete, simplicial, bushy $\mathbb{R}$--trees is
a bounded distortion equivalence, 
then $h$ is PQ-symmetric.
\end{teorema}

\begin{proof}
Let $K>0$ be a bushy constant for both $T$ and $T'$, and
let $\e^A$ be the constant for the
bounded distortion equivalence, where $A\geq 0$.
Let $p = A+2K$, $q_1=4K+3A$,
$q_2=\max\{ q_1^{1/p}, q_1\e^{2A}\}$, and
$q=\max\{ q_1, q_2\}$.
We will show that $h$ is  PQ-symmetric with constants $p$ and $q$.

Consider any three points $G_0,G_1,G_2\in end(T,v)$ and suppose
$d(G_0,G_1)=d(G_1,G_2)\geq d(G_0,G_2)$. Let $t_0=(G_0|G_1)_v$ and
$t_1=(G_0|G_2)_v$. Then, $d(G_0,G_1)=t\cdot d(G_0,G_2)$ with
$t=\e^{t_1-t_0}\geq 1$. Since $T$ is a  simplicial $\br$-tree,
there exists $k\in\bn$  and vertices
$x_0, \dots, x_{k+1}\in T$ as shown in
Figure~\ref{fig:end homeo} such that
\begin{enumerate}
	\item $[G_0|G_1]=[v,x_0]$ and $[G_0|G_2]=[v,x_{k+1}]$
	\item $d(x_0,x_{k+1})=k+1$
	\item $d(v,x_i) = d(v,x_0) +i\in\bn$ for each $i=0,\dots, k+1$
	\item $x_i\in [v, x_{k+1}]$ for each $i=0,\dots, k+1$
\end{enumerate}
For each $i=0,\dots, k+1$, let $\mathcal{F}_i=\{F\in end(T,v) \ | \ [F|G_0]=[v,x_i] \}$.
Note that $\mathcal{F}_i$ may be empty.
\begin{figure}
\centering 
\includegraphics[scale=0.25]{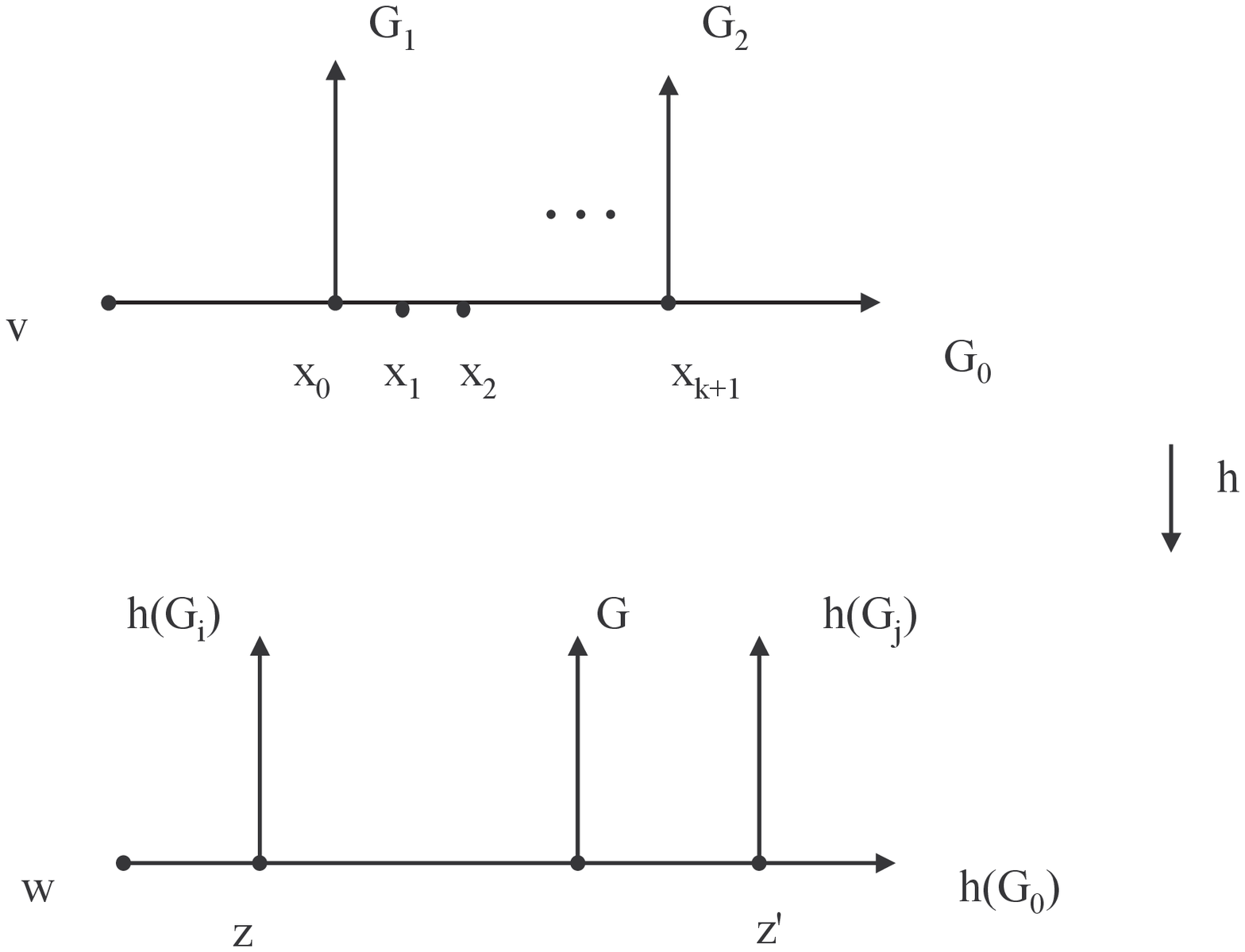}
\caption{$h\co end(T,v) \to end(T',w)$.}
\label{fig:end homeo}
\end{figure}
The first step is to show that 
$$d(h(G_0),h(G_1) \leq q_1\cdot t^p d(h(G_0),h(G_2)).$$
Suppose on the contrary that 
$$d(h(G_0),h(G_1)>(4K+3A)\cdot
t^{A+2K}d(h(G_0),h(G_2)),$$ 
which is to say,
$$(h(G_0)|h(G_1))_w-(h(G_0)|h(G_2))_w>(k+1)(A+ 2K)+4K+3A.$$ 
For any
two $F_1,F_2 \in \mathcal{F}_i$, $0\leq i \leq k+1$, 
the bounded distortion condition with respect to
$G_0$ implies that 
$$|(h(G_0)|h(F_1))_w-(h(G_0)|h(F_1))_w|\leq A.$$
It follows that for each $i=0,\dots, k+1$, there exists 
a subinterval $I_i$ of $h(G_0)$ of length $A$ such that
the bifurcation points of $h(G_0)$ and elements of
$h(\mathcal{F}_i)$ and  are all contained in the subinterval $I_i$.
Since $T'$ is bushy with constant $K$, there exists
$G\in end(T',w)\setminus\bigcup_{i=0}^{k+1}h(\mathcal{F}_i)$ 
such that \[(h(G_0)|h(G_1))_w+A<(G|h(G_0))_w<(h(G_0)|h(G_2))_w-A.\]
Suppose $z$ is the bifurcation point of $h(G_0)$ and $h(G_1)$,
and $z'$ the bifurcation point of $h(G_0)$ and $h(G_2)$. 
The interval $I := [G_0(||z||+A),G_0(||z'||-A)]$ has length at
least $(k+1)(A+ 2K)+4K+A$. Thus, 
$I$ contains an open interval of length $2K$ disjoint from each
of the $k+2$ subintervals $I_i$, for $0\leq i\leq k+1$.
Since $K$ is the bushy constant, there exists $H\in end(T',w)$ such that
the bifurcation point of $h(G_0)$ and $H$ is in $I$.
The surjectivity of $h$ implies there exists 
$G\in end(T,v)\setminus \bigcup_{i=0}^{k+1}\mathcal{F}_i$ such that $h(G)=H$.
There are two cases:
\begin{itemize} \item[(a)]
$(G|G_0)<||x_0||$. This leads to contradiction applying the
bounded distortion condition on $G$ with respect to $G_0,G_1$
in conjunction with the fact that $|(h(G)|h(G_0))_w-(h(G)|h(G_1))_w|> A$.
\item[(b)] $(G|G_0)>||x_{k+1}||$. This leads to a contradiction 
by applying the bounded distortion condition on
$G_2$ with respect to $G,G_0$ in conjunction with the fact that 
$(H|h(G_0))_w< (h(G_0)|h(G_2))_w-A$.
\end{itemize}
In either case, there is a contradiction. The conclusion is that
$d(h(G_0),h(G_1))\leq q_1 t^p \cdot d(h(G_0),h(G_2))$
and the first step is complete.

The second and final step is to show that
$$d(h(G_0),h(G_2)) \leq q_2\left(\frac{1}{t}\right)^{\frac{1}{p}} \cdot
d(h(G_0),h(G_1)),$$ 
which is to say, 
$$d(h(G_0),h(G_1))\geq \frac{1}{q_2}
t^{\frac{1}{p}} \cdot d(h(G_0),h(G_2)).$$
To this end note that the bounded distortion assumption together with the fact that
$d(G_0,G_1)=d(G_1,G_2)\geq d(G_0,G_2)$, implies that
\[d(h(G_0),h(G_1))\geq \e^{-A} \cdot d(h(G_0),h(G_2)).\]
Let $t'_0=(h(G_0)|h(G_1))_v$, $t'_1=(h(G_0)|h(G_2))_v$ and
$t'=\e^{|t_1-t_0|}$. By the same argument as before,
$(4K+3A)t'^{p}\geq t$. Therefore, $t'\geq
(\frac{t}{4K+3A})^{\frac{1}{p}}$. In other words, either
\[d(h(G_0),h(G_1))\geq \Big(\frac{1}{4K+3A}\Big)^{\frac{1}{p}}
t^{\frac{1}{p}}d(h(G_0),h(G_2))\] or \[\e^A d(h(G_0),h(G_1)) \geq
d(h(G_0),h(G_2))\geq \Big(\frac{1}{4K+3A}\Big)^{\frac{1}{p}}
t^{\frac{1}{p}} d(h(G_0),h(G_1)),\] where the second is only possible
if $t^{\frac{1}{p}}\leq (4K+3A)\e^{A}$.
In the first case, the desired result follows from the fact that
$q_2\geq (4K+3A)^{\frac{1}{p}}$.
In the second case, \[d(h(G_0),h(G_1)) \geq \e^{-A} d(h(G_0),h(G_2))\geq
\e^{-A} \frac{t^{\frac{1}{p}}}{(4K+3A)\e^{A}}d(h(G_0),h(G_2)).\]
Thus, the desired result  follows from the fact that $q_2\geq (4K+3A)\e^{2A}.$
In both cases, it is readily seen that $d(h(G_0),h(G_1))\geq
\frac{1}{q_2} t^{\frac{1}{p}}d(h(G_0),h(G_2))$, completing the second step.
Thus, $h$ is PQ-symmetric. 
\end{proof}

\begin{definition} A map between metric spaces $f\co X\to Y$ is
\emph{roughly isometric} if there is a constant $a>0$ such that
$d(x,x')-a\leq d(f(x)f(x'))\leq d(x,x')+a$ for all $x,x'\in X$.
If, in addition, $f(X)$ is a net in $Y$, then $f$ is  a
\emph{rough isometry} and $X$ and $Y$ are  {\it roughly isometric}
\end{definition}

Obviously, a rough isometry is a quasi-isometry.

\begin{prop}\label{simplicication} If $T$ is an $\mathbb{R}$-tree $T$, then
there exists a  simplicial $\mathbb{R}$-tree $S$ such that $T$ and $S$ are
(continuously) rough isometric.
In particular, $T$ and $S$ are (continuously) quasi-isometric.
\end{prop}

\begin{proof} 
Fix a root $v\in T$ and consider $(T,v)$ the rooted $\mathbb{R}$-tree. 
For each $n=1,2,3,\dots$, let
$S_n = \partial B(v,n)$.
Let $S_0=\{ w\}$, where $w$ is a point not in $\bigcup_{i=1}^\infty S_n$.
Define the rooted, simplicial $\br$-tree $(S,w)$ by joining with edges of length 1 all
the points in $S_1$ to  $w$ and each point $s \in S_k$ to the
unique point $t\in S_{k-1}$ such that $t\in [v,s]$. Thus, $\bigcup_{i=0}^\infty S_i$ is the vertex set of $S$.

Any  isometric embedding of the form
$F\co [0,\infty)\to T$ or $F\co [0,N]\to T$, where $F(0) = v$ and $N\in \bn$,
is completely determined by the sequence $\{ F(k)\}_{k=1}^\infty$
or $\{ F(k)\}_{k=1}^N$, respectively.
Since these sequences are in the vertex set of $S$, there is a corresponding
isometric embedding of the form
$F'\co [0,\infty)\to S$ or $F'\co [0,N]\to S$, where $F'(0) = w$.
Extend this construction to isometric embeddings of the form
$F\co [0,t]\to T$, where $F(0)=v$ and $t>0$, by letting
$F'$ denote $(F|[0,\lfloor t\rfloor ])'$,
where $\lfloor t\rfloor$ denotes the greatest integer less than or equal  to $t$.
Denote this association by $p\co F\mapsto F'$.
(Note that $p$ is surjective between the sets of isometric embeddings. However, $p$ need not be injective  because of the existence of domains of the
form $[0,t]$ for non-integral $t$.) 
Define the map $h:(T,v) \to (S,w)$ as follows: 
First, $h(x)=w$ for all $x \in B(v,1)$. Second, $h(F(t))=F'(t-1)$ for any isometric embedding 
of the form
$F\co [0,\infty)\to T$ or $F\co [0,t]\to T$, where $F(0) = v$ and $t\geq 1$. 
In particular, for every integer $k>0$ in the domain of $F$, $h(F(k))=F'(k-1)$ which, by the definition of $S$, is the point $F(k-1)$. 
In order to show that $h$ is well-defined, let
$F$ and $G$ be two isometric embeddings into $T$ of the type considered above.
If $t$ is in the domains of $F$ and $G$, then  $F(t)=G(t)$ implies that $F(\lfloor t\rfloor)=G(\lfloor t\rfloor)$.
In particular, $F'(t-1)=G'(t-1)$. Therefore, $h(F(t))=F'(t-1)=G'(t-1)=h(G(t))$. 

Finally, define $j:(S,w) \to (T,v)$ by
$j(H(t))=F(t)$, where $H\co [0,\infty)\to S$ or $H\co [0,N]\to S$ is an isometric embedding such that $H(0)=w$ and $N\in\bn$,
and $F$ is an isometric embedding in $T$ such that $p(F)=H$.
Note that $j(H(t))$ does not
depend on the choice of $F\in p^{-1}(H)$. For  $p(F)=p(G)$ and $F\not= G$ if and only if $F$ and $G$ have finite domains and
$F(k)=H(k)=G(k)$ and for all integers $k$ in their domains.
Clearly, $h$ is a roughly isometric map with with constant $a=1$. Moreover, $j$ shows that $h$ is a  rough isometry between
$T$ and $S$ .
\end{proof}

The following definition is well-known. The name ``doubling'' comes from the fact that  usually $C=2$. The version
given here appears in Bonk and Schramm \cite[Section 9]{BoSchr}.

\begin{definition} A metric space is  \emph{doubling} if for every $C>1$ there exist $N\in \bn$ such that:
if $0< r< R$ with $R/r=C$, then  every open ball of radius $R$ in $X$ can be covered by $N$ open balls of radius $r$.
\end{definition}

\begin{definition} A metric space is  \emph{pseudo-doubling} if for every $C>1$ there exist $N\in \bn$ such that:
if $0< r< R$ with $R/r=C$ and $x\in X$, then there are at most $N$ balls $B$ such that $B(x,r)\subseteq B \subseteq B(x,R)$.
\end{definition}

\begin{remark}\label{remark:pseudo-doubling} 
The following facts are easy to verify.
\begin{enumerate}
	\item Every doubling ultrametric space is pseudo-doubling.
	\item Let $X=end(T,v)$, where $(T,v)$ is a rooted, geodesically complete, simplicial $\br$-tree. 
	Then $X$ is pseudo-doubling. However, if the set of valences of the vertices
	of $T$ 	is unbounded, then $X$ is  not doubling. 
		\item The closed unit interval $[0,1]$ with its standard metric is a doubling metric space that is not pseudo-doubling. 
\end{enumerate}
\end{remark}

\begin{corollary}\label{cor:BDE PQ for pd}
If $f\co X\to Y$ is a bounded distortion equivalence between bounded, complete, uniformly perfect, 
pseudo-doubling ultrametric spaces, then $f$ is a $PQ$-symmetric homeomorphism. 
\end{corollary}

\begin{proof} Let $X=end(T_X,v), Y=end(T_Y,w)$, where $T_X,T_Y$ are geodesically complete, bushy $\br$-trees (see Remark \ref{bushy-unif}). Since $X$ is pseudo-doubling, for $R/r=\e$ there exists a constant $N$ such that for any $x\in X$ there are at most $N$ balls $B$ such that $B(x,r)\subseteq B \subseteq B(x,R)$. 
It follows that
for each $k\in \bn$ there are at most $N$ different real numbers $\alpha_1,...,\alpha_n$ with $k< \alpha_1<\alpha_2<...<\alpha_n< k+1$ such that 
whenever $F,G \in end(T_X,v)$ and  $k<(F|G)_v<k+1$, then $(F|G)_v=\alpha_i$ for some $i=1,\dots,n$. 

We will now distort the tree $T_X$ up to rooted homeomorphism so that Remark~\ref{remark:rooted homeo} applies.
Divide each interval $[k,k+1]$ into $N+1$ subintervals of length $1/(N+1)$. For any pair of points $F,G \in end(T_X,v)$, if $(F|G)_v=\alpha_i$,
change it to $(F|G)_v=k+i/(N+1)$. Denote this new tree by $[T_X]$. It is immediate to see that 
the natural rooted homeomorphism $T_X\to [T_X]$ induces a 
bi-Lipschitz equivalence $end(T_X,v)\to end([T_X],v)$  with constant $e$.

Now consider $S_X$, a barycentric subdivision of $[T_X]$ obtained by  dividing each edge into $N+1$ subintervals of length $1/(N+1)$ and adding the corresponding $N$ new vertices. Consider each new edge to be of length 1. Then $S_X$ is similar to $[T_X]$ (the distance is multiplied by $N+1$) and, by construction of $[T_X]$, $S_X$ is simplicial. Also, the canonical map $h_X:end(T_X,v) \to end(S_X,v)$ is such that $(\frac{1}{\e} d(F,G))^{N+1}\leq d(h_X(F),h_X(Y))\leq (\e \cdot d(F,G))^{N+1}$. Thus, $h_X$ is PQ-symmetric.

Likewise distort $T_Y$ to $[T_Y]$ and construct $S_Y$ and $h_Y:end(T_Y,w) \to end(S_Y,w)$.

It follows that  $f$  induces a bounded distortion equivalence 
$$\tilde{f}:= h_Y\circ f \circ h_X^{-1} \co end(S_X,v) \to end(S_Y,w)$$ 
between ends of rooted, geodesically complete, simplicial, bushy $\br$-trees. 
Theorem~\ref{prop:BQC implies PQ} implies that $\tilde{f}: end(S_X,v) \to end(S_Y,w)$ is  PQ-symmetric. 
Since $f=h_Y^{-1}\circ \tilde{f} \circ h_X$,  Remark~\ref{rem:TV inverses} implies that $f$ is PQ-symmetric.
\end{proof}

\section{The examples}
\label{sec:examples}

This section contains several examples that illustrate the sharpness of the results in the previous sections.
In addition, some of the examples answer several questions raised by Mirani \cite{Mirani2006, Mirani2008}.
The basic building block for our examples is the infinite binary tree. We begin by introducing notation that will help us describe the examples.

\begin{notation}
\label{notation:examples}
Let $(T_2,v)$ denote the rooted, infinite binary tree, also known as the Cantor tree. 
Thus, $T_2$ is a locally finite, simplicial $\br$-tree, the root has valency two, and all other roots have valency three.
All edges are labeled $0$ or $1$; every vertex is incident to at least one edge labeled $0$ and to at least one edge labeled $1$.
Let $X_2= end(T,v)$. Thus, 
$$X_2= \{ x=(x_1,x_2, \dots) ~|~ x_i\in \{0,1\}, i=1,2,\dots\} = \prod_1^\infty\{0,1\}$$
with Gromov product at infinity given by
$(x|y)_v = i$ if and only if $x_j=y_j$ for $j\leq i$ and $x_{i+1}\not= y_{i+1}$.

If $a=(a_1, \dots, a_n)\in\prod_1^n\{ 0,1\}$, then
$$aX_2:= \{ y\in X_2 ~|~ y_i = a_i ~\text{for $1\leq i\leq n$}\}.$$
Thus, each $y\in aX_2$ can be written uniquely as $y=ax$, where $x\in X_2$. Note that $aX_2$ is a closed ball in $X_2$.

We denote certain infinite and finite sequences as follows:
\begin{align*}
	\bar 0 =& (0,0,0,\dots) \in X_2 &
	\bar 1 =& (1,1,1,\dots) \in X_2 \\
	\bar 0_n =& (0,\dots,0) \in \prod_1^n\{ 0,1\} &
	\bar 1_n =& (1,\dots,1) \in \prod_1^n\{ 0,1\}
\end{align*}
If $n=0$, then $\bar 0_n$ and $\bar 1_n$ denote the empty sequences.

Since much of Ghys and de la Harpe \cite{GhysdelaHarpe} and Mirani \cite{Mirani2006, Mirani2008}
is in the setting of simplicial $\br$-trees with the property that each vertex has valency at least $3$,
we want to provide examples in that class of trees whenever possible. A simple way to achieve this is to let
$(T_3, v))$ be the unique rooted, simplicial $\br$-tree containing $(T_2,v)$ as a rooted subtree such that
every vertex of $T_3$ has valency exactly $3$. In other words, $T_3$ is the infinite $3$-regular tree.

Let $X_3 = end(T_3,v)$. Note that $X_2$ and $X_3$ are compact, uniformly perfect, doubling ultrametric spaces.
Moreover, $X_2$ is a closed and open subspace of $X_3$.
\end{notation}

The first example shows that the hypothesis of Theorems~\ref{thm:main} and \ref{prop:BQC implies PQ} can not be changed from
``bounded distortion equivalence'' to ``bi-H\"older and conformal homeomorphism.''

\begin{ejp}\label{quasi-conf} 
There exist a bi-Hölder, conformal homeomorphism $h\co X_3\to X_3$ that does not 
have bounded distortion. In particular, $h$ is not  PQ-symmetric. 
\end{ejp}

\begin{proof}
For each $i\geq 1$, define the following sequences:
\begin{align*}
a_i =& \bar 1_{i-1}011\\
b_i =& \bar 1_{i-1}\bar0_{i+1}11\\
c_i =& \bar 1_{i-1}010\\
F_i =& \bar 1_{i-1}\bar 0	
\end{align*}
Define a homeomorphism $h\co X_3\to X_3$ by  
$$\begin{cases}
h(a_ix) = b_ix & \text{for all $x\in X_2$ and $i=1,2,\dots$}\\
h(b_ix) = a_ix & \text{for all $x\in X_2$ and $i=1,2,\dots$}\\
h(x) = x & \text{for all $x\in X_3\setminus\bigcup_{i=1}^\infty(a_iX_2\cup b_iX_2)$}
\end{cases}$$
To see that $h$ is conformal, note that if $\bar 1\not= F\in X_2$, then there exists $n_F\geq 1$ such that
the ${n_F}^{th}$ term of $F$ is $0$. It follows that the image under $h$ of any sphere centered at $F$ of radius 
$r \leq \e^{-1-n_F}$ is a sphere; thus, $D_h(F,r)=1$. Clearly, for all $r>0$, $D_h(\bar 1,r)=1$  and $D_h(F,r)=1$ whenever
$F\in X_3\setminus X_2$. Thus, $h$ is conformal.

To see that $h$ is bi-H\"older, 
it suffices to check that for any
pair of points $F,G\in X_3$, $\frac{1}{2} (F|G)_v -1 \leq
(h(F)|h(G))_v \leq 2(F|G)_v +1$. It only happens that 
$(h(F)|h(G))_v \neq (F|G)_v$  when 
\begin{enumerate}
	\item one of $F$ or $G$ is in $a_iX_2$ or $b_iX_2$ and the other is in $X_3 \setminus\bigcup_{i=1}^\infty(a_iX_2\cup b_iX_2)$, or
	\item  $F,G\in a_iX_2$ for some $i\geq 1$, or
	\item $F,G\in b_iX_2$ for some $i\geq 1$. 
\end{enumerate}
In each case, the required inequalities are easy to check.

To see that $h$ does not have bounded distortion, choose $G_i\in a_iX_2$ and $H_i\in c_iX_2$
for each $i\geq 1$. Note that $h(F_i) =F_i$, $h(G_i)\in b_iX_2$, and $h(H_i) = H_i$ for all $i\geq 1$.
Moreover, $(F_i|G_i)_v=(F_i|H_i)_v=i$ and $(h(F_i)|h(G_i))_v= 2i$ for all $i\geq 1$.
Thus, $|(h(F_i|h(G_i)_v-(h(F_i|h(H_i))_v| =i$ and Proposition~\ref{prop:bqc gpi} shows that
$h$ does not have bounded distortion.

Finally, it follows from Proposition~\ref{prop:QS implies BQE} that $h$ is not PQ-symmetric.
\end{proof}

The following example illustrates that the result of Tukia and V{\"a}is{\"a}l{\"a} \cite[Theorem 3.14]{TukVais} mentioned in
Remark~\ref{remark:PQ biH} above, 
that a PQ-symmetric homeomorphism  between bounded metric spaces is bi-H\"older, does not hold  if the 
hypothesis ``PQ-symmetric'' is weakened to ``quasiconformal,'' even for  ultrametric spaces  as 
nice as $X_3$.

\begin{ejp}
\label{ex:qc not biholder} 
There exists a conformal homeomorphism $h\co X_3 \to X_3$ that is not 
bi-H\"older.
\end{ejp}

\begin{proof}
For each $i\geq 1$, define the following sequences:
\begin{align*}
g_i =& \bar 1_{i-1}01\\
h_i =& \bar 1_{i-1}\bar0_{i^2-i}1
\end{align*}
Define a homeomorphism $h\co X_3\to X_3$ by  
$$\begin{cases}
h(g_ix) = h_ix & \text{for all $x\in X_2$ and $i=1,2,\dots$}\\
h(h_ix) = g_ix & \text{for all $x\in X_2$ and $i=1,2,\dots$}\\
h(x) = x & \text{for all $x\in X_3\setminus\bigcup_{i=1}^\infty(g_iX_2\cup h_iX_2)$}
\end{cases}$$
The argument to prove that $h$ is conformal is similar to
the one in Example~\ref{quasi-conf}. 
To see that $h$ is not bi-Hölder, for each $i=1,2,3,\dots$ choose $G_i\in g_iX_2$
and let $F_i\in X_2$ be as
in Example~\ref{quasi-conf}.
Note that $h(F_i)=F_i$ and $h(G_i)\in h_iX_2$ for all $i\geq 1$.
Thus, $d(h(F_i),h(G_i))=\e^{-i^2}$ while
$d(F_i,G_i)=\e^{-i}$ for all $i\geq 1$. It follows that $h$ is not bi-Hölder.
\end{proof}

In the converse direction, Mirani \cite{Mirani2006, Mirani2008}
speculated that bi-H\"older homeomorphisms on spaces such as $X_3$ might be
quasi-conformal. The following example shows this is not the case.

\begin{ejp} 
There exist a  bi-H\"older homeomorphism $h\co X_3\to X_3$ that is not 
not quasiconformal.
\end{ejp}

\begin{proof}
For each $i\geq 1$, define the following sequences:
\begin{align*}
g_i =& \bar 0_{4^i}11\\
h_i =& \bar 0_{2\cdot 4^i}11
\end{align*}
Define a homeomorphism $h\co X_3\to X_3$ by  
$$\begin{cases}
h(g_ix) = h_ix & \text{for all $x\in X_2$ and $i=1,2,\dots$}\\
h(h_ix) = g_ix & \text{for all $x\in X_2$ and $i=1,2,\dots$}\\
h(x) = x & \text{for all $x\in X_3\setminus\bigcup_{i=1}^\infty(g_iX_2\cup h_iX_2)$}
\end{cases}$$
The argument to prove that $h$ is bi-Hölder is similar to
the one in Example~\ref{quasi-conf}: one checks that
$F,G\in end(T,v)$, $\frac{1}{2} (F|G)_v -1 \leq (h(F)|h(G))_v \leq
2(F|G)_v +1$ for every $F, G\in X_3$.

To see that $h$ is not quasiconformal, it suffices to show that
$\underset{\epsilon\to 0}\limsup\, D_h\left(\bar 0,\epsilon\right) =\infty$.
To this end, let $K>0$ and  $N\in
\mathbb{N}$ be given. 
There exists $i_0>N$ such  $4^{i_0}>K$.
Choose $G_1\in g_{i_0}X_2$ and $G_2\in\bar 0_{4^{i_0}}10X_2$.
It follows that $h(\bar 0) =\bar 0$, $h(G_1)\in h_{i_0}X_2$, and $h(G_2) = G_2$.
Therefore, $(h(F_0)|h(G_2))_v=4^{i_0}$ and
$(h(F_0)|h(G_1))_v=2\cdot 4^{i_0}$.
Hence,
$D_h\left(\bar 0, \e^{-4^{i_0}}\right) \geq \frac{\e^{-4^{i_0}}}{\e^{-2\cdot4^{i_0}}} = \e^{4^{i_0}}.$
\end{proof}

For functorial reasons, Mirani \cite{Mirani2006, Mirani2008}
was interested in compositions of bi-H\"older, quasiconformal homeomorphisms
between spaces such as $X_3$.
The following example illustrates that these compositions need not be well-behaved.
This is in contrast to the situation for PQ-symmetric homeomorphisms as observed 
by Tukia and V{\"a}is{\"a}l{\"a} \cite[Theorem 2.2]{TukVais} (see Remark~\ref{rem:TV inverses}).

\begin{ejp} 
There exist two bi-H\"older, quasiconformal homeomorphisms 
$$h_1, h_2\co X_3\to X_3$$ such that 
the composition $h_2\circ h_1$ is not quasiconformal.
Moreover, $h_1$ and $h_2$ are bounded distortion equivalences and $h_2$ is conformal.
\end{ejp}

\begin{proof}
For each $i\geq 1$, define the following sequences:
\begin{align*}
g_i^a =& \bar 0_{2i}11\\
h_i^a =& \bar 0_{2i+1}11\\
g_i^b =& \bar 0_{4i}1\\
h_i^b =& \bar 0_{8i}1
\end{align*}
Define a homeomorphism $h_1\co X_3\to X_3$ by  
$$\begin{cases}
h_1(g_i^ax) = h_i^ax & \text{for all $x\in X_2$ and $i=1,2,\dots$}\\
h_1(h_i^ax) = g_i^ax & \text{for all $x\in X_2$ and $i=1,2,\dots$}\\
h_1(x) = x & \text{for all $x\in X_3\setminus\bigcup_{i=1}^\infty(g_i^aX_2\cup h_i^aX_2)$}
\end{cases}$$
Define a homeomorphism $h_2\co X_3\to X_3$ by  
$$\begin{cases}
h_2(g_i^bx) = h_i^bx & \text{for all $x\in X_2$ and $i=1,2,\dots$}\\
h_2(h_i^bx) = g_i^bx & \text{for all $x\in X_2$ and $i=1,2,\dots$}\\
h_2(x) = x & \text{for all $x\in X_3\setminus\bigcup_{i=1}^\infty(g_i^bX_2\cup h_i^bX_2)$}
\end{cases}$$
It is readily seen that $h_1$ is $\e$-quasiconformal. Indeed, $d_{h_1}(F,\epsilon)\leq \e$ for all
$F\in X_3$ and for all $\epsilon >0$, from which it also follows that $h_1$ is 
a  bounded distortion equivalence. Likewise,
$d_{h_2}(F,\epsilon)=1$ for all
$F\in X_3$ and for all $\epsilon >0$, from which it also follows that $h_2$ is 
conformal and a  bounded distortion equivalence.
\end{proof}

The  next two examples show that ``uniformly perfect'' can not be weakened to ``perfect'' 
in Theorem~\ref{thm:main} and Corollary~\ref{cor:BDE PQ for pd}, and that the ``bushy'' condition is
needed in Theorem~\ref{prop:BQC implies PQ}.

\begin{ejp} 
\label{example:BQE not QS}
There exist a compact, perfect, pseudo-doubling, ultrametric space $Z$ and a conformal homeomorphism $h\co Z\to Z$ such that $h$ is a bounded distortion equivalence, but $h$ is not quasi-symmetric. Moreover,  $h$ can be chosen to be bi-H\"older or not bi-H\"older.
\end{ejp}

\begin{proof}
We begin by constructing compact perfect, perfect, ultrametric spaces $X$ and $Y$ and a conformal, bounded distortion equivalence
$f\co X\to Y$ such that $f$ is neither quasi-symmetric nor bi-H\"older.
Let $(T_2^i, v_i)$ denote a copy of the rooted, infinite binary tree $(T,v)$ for each $i\geq 1$.
Form the tree $T= T_2\bigvee_{i=1}^\infty T_2^i$ by attaching, for each $i\geq 1$, $T_2^i$ to $T_2$ by identifying $v_i$ with the vertex on $\bar 1\subseteq T_2$ that is a distance $i$ from $v$.
Note that $(T,v)$ is a rooted, geodesically complete, simplicial $\br$-tree. 

For each $i\geq 1$, let $X_2^i = end(T_2^i,v_i)$. Then 
$$X := end(T,v) = X_2 \bigcup_{i=1}^\infty\bar 1_iX_2^i.$$
As before, the points of $X_2$ are infinite sequences of $0$'s and $1$'s. For each $i\geq 1$, the points of $X_2^i$ are
denoted by infinite sequences of $0_i$'s and $1_i$'s.
A typical point of $X$ is either a point of $X_2$ or of the form $\bar 1_ix= (\underbrace{1,\dots, 1}_{i})x$, where $x\in X_2^i$.
Note that $X$ is a compact, uniformly perfect, pseudo-doubling, ultrametric space.

For each $i\geq 1$ define
$$Y_2^i = \{ y\in X_2^i ~|~ y =(\underbrace{x_1,\dots, x_1}_{i},\underbrace{x_2,\dots, x_2}_{i},\underbrace{x_3,\dots, x_3}_{i},\dots),
x_j\in\{0_i,1_i\} \text{ for each $j\geq 1$}\}.$$
Note that for each $i\geq 1$, $Y_2^i$ is a closed, uniformly perfect subset of $X_2^i$.
There is an evident subtree $S_2^i$ of $T_2^i$ such that $end(S_2^i,v_i) = Y_2^i$. The vertices of $S_2^i$ of valency $3$ are those vertices of $T_2^i$ 
that are a distance $ni$ from $v_i$ for some $n=1,2,3,\dots$.

Let $T'= T_2\bigvee_{i=1}^\infty S_2^i$ so that $(T',v)$ is a rooted, geodesically complete subtree of $(T,v)$.
For each $i\geq 1$, let $Y_2^i = end(S_2^i,v_i)$. Then 
$$Y := end(T',v) = X_2 \bigcup_{i=1}^\infty\bar 1_iY_2^i.$$
Note that $Y$ is a compact, perfect, pseudo-doubling, ultrametric space.
The simplicial $\br$-tree $T'$ is not bushy because for any $K>0$, if $i\geq 2K$, then there are points $x$ in $S_2^i$ such that if 
$y$ is within $K$ of $x$, then $T'\setminus \{y\}$ has exactly $2$ components. In particular, $Y$ is not uniformly perfect. 

Define $f\co X\to Y$ by
$$f(\bar 1_ix) = \begin{cases}
x & \text{if $i=0$ and $x\in X_2$}\\
\bar 1_i(\underbrace{x_1,\dots,x_1}_{i}, \underbrace{x_2,\dots,x_2}_{i}, \underbrace{x_3,\dots,x_3}_{i},\dots) & \text{if $i\geq 1$ and $x\in X_2^i$}
\end{cases}
$$
Note that $f(X_2) = Y_2$ and $f(X_2^i) = Y_2^i$ for all $i\geq 1$.

Clearly, $f$ is induced by a rooted homeomorphism $\hat{f}\co (T,v)\to (T',v)$; i.e.,
$f(F)= \hat{f}\circ F$ for each $F\in end(T,v)$. 
It follows from Remark~\ref{remark:rooted homeo} that
$f$ is conformal
and
$f$ is a bounded distortion equivalence.

To see that $f$ is not quasi-symmetric, choose $G\in X_2$ and $F_i, H_i\in X_2^i$ such that
$(F_i|H_i)_v= i+1$ for each $i\geq 1$.
Then $|(F_i|H_i)_v-(F_i|G)_v|=1$ and 
$$|(f(F_i)|f(H_i))_v-(f(F_i)|f(G))_v|=i$$ for each $i\geq 1$.
Proposition~\ref{prop:qs gpi} implies that $f$ is not quasi-symmetric.

To verify that $f$ is not bi-H\"older, 
choose for each $i\geq 1$, $F_i, G_i\in X_2^i$ such that $(F_i|G_i)_v=2i$. Then
$(f(F_i)|f(G_i))_v = i+i^2$. It follows that the criterion in  
Remark~\ref{remark:gpi biH} is violated. 

We have now constructed a conformal, bounded distortion equivalence $f\co X\to Y$ that is neither quasi-symmetric nor bi-H\"older.
In order, to get an example in which the domain and range are the same, 
let $Z= end(T\vee T', v)$ so that $Z= X\cup Y$ and define $h\co Z\to Z$ by $h_X= f$ and $h|Y = f^{-1}$.

Finally, we briefly indicate the modifications that need to be made in order to get a bi-H\"older example.
For each $i\geq 1$ replace $Y_2^i$ by
$$\hat{Y}_2^i = \{ y\in X_2^i ~|~ y =(\underbrace{x_1,\dots, x_1}_{i},x_2,x_3,\dots),
x_j\in\{0_i,1_i\} \text{ for each $j\geq 1$}\}$$
and replace  $f$ by
$$\hat{f}(\bar 1_ix) = \begin{cases}
x & \text{if $i=0$ and $x\in X_2$}\\
\bar 1_i(\underbrace{x_1,\dots,x_1}_{i}, x_2,x_3,\dots) & \text{if $i\geq 1$ and $x\in X_2^i$}
\end{cases}
$$
These changes are enough to construct a bi-H\"older, conformal, bounded distortion equivalence  that is not quasi-symmetric.
\end{proof}

\begin{ejp} 
\label{example:QS not PQS}
There exist a compact, perfect, pseudo-doubling, ultrametric space $Z$ and a conformal homeomorphism $h\co Z\to Z$ such that $h$ is quasi-symmetric, but $h$ is not PQ-symmetric. Moreover,  $h$ can be chosen to be bi-H\"older or not bi-H\"older.
\end{ejp}

\begin{proof}
These examples are modifications of Example~\ref{example:BQE not QS}. Therefore, we will just briefly indicate the changes that need to be made.

We begin by showing how to produce the example that is not bi-H\"older. Let $Y$ be as in Example~\ref{example:BQE not QS}.
Construct a space $W$ as $Y$ is constructed above---except that $Y_2^i$ is replaced by
$$\widetilde{Y}_2^i = \{ y\in X_2^i ~|~ y =(\underbrace{x_1,\dots, x_1}_{i^2},\underbrace{x_2,\dots, x_2}_{i^2},\underbrace{x_3,\dots, x_3}_{i^2},\dots),
x_j\in\{0_i,1_i\} \text{ for each $j\geq 1$}\}.$$
Let $$\widetilde{Y}
  = X_2 \bigcup_{i=1}^\infty\bar 1_i\widetilde{Y}_2^i$$
and define $\tilde{f}\co Y\to \widetilde{Y}$ by
$\tilde{f}(x) = x$ if $x\in X_2$
and
$$\tilde{f}(\bar 1_i(\underbrace{x_1,\dots,x_1}_{i}, \underbrace{x_2,\dots,x_2}_{i}, \underbrace{x_3,\dots,x_3}_{i},\dots))
= 
\bar 1_i(\underbrace{x_1,\dots,x_1}_{i^2}, \underbrace{x_2,\dots,x_2}_{i^2}, \underbrace{x_3,\dots,x_3}_{i^2},\dots)$$
if $i\geq 1$ and $(x_1,x_2,x_3,\dots) \in X_2^i$.
It follows from Remark~\ref{remark:rooted homeo} that
$\tilde{f}$ is conformal.
One can verify that $\tilde{f}$ is quasi-symmetric
by invoking Proposition~\ref{prop:qs gpi} with $$\gamma(t) =\begin{cases} t^2 & \text{if $t\geq 0$}\\
t & \text{if $t<0$.}\end{cases}$$
To verify that $f$ is not bi-H\"older, 
choose for each $i\geq 1$, $F_i, G_i\in Y_2^i$ such that $(F_i|G_i)_v=2i$. Then
$(\tilde{f}(F_i)|\tilde{f}(G_i))_v = i+i^2$. It follows that the criterion in  
Remark~\ref{remark:gpi biH} is violated. 
To see that $\tilde{f}$ is not PQ-symmetric, choose $G\in X_2$ and $F_i, H_i\in Y_2^i$ such that
$(F_i|H_i)_v= 2i$ for each $i\geq 1$.
Then $|(F_i|H_i)_v-(F_i|G)_v|=i$ and 
$$|(\tilde{f}(F_i)|\tilde{f}(H_i))_v-(\tilde{f}(F_i)|\tilde{f}(G))_v|=i^2$$ for each $i\geq 1$.
Proposition~\ref{prop:pqs gpi} implies that $f$ is not PQ-symmetric.
The passage from $\tilde{f}\co Y\to\tilde{Y}$ to $h\co Z\to Z$ takes place as in Example~\ref{example:BQE not QS}.

To construct the bi-H\"older example, 
replace $\tilde{f}\co Y\to\tilde{Y}$ by $\tilde{f}'\co Y'\to\widetilde{Y}'$
with the following explanations.
First, just as $Y$ arises by attaching trees $S_2^i$ to points on $\bar 1$ a distance $i$ from $v$, the space $Y'$ arises by
attaching those trees to points a distance $i^2$ from $v$.
Second, just as $\widetilde{Y}$ arises from a modification of $Y$ by replacing each $Y_2^i$ by $\widetilde{Y}_2^i$, the space
$\widetilde{Y}'$ arises by replacing $Y_2^i\subseteq Y'$ by
$$\left(\widetilde{Y}_2^i\right)' = \{ y\in X_2^i ~|~ y =(\underbrace{x_1,\dots, x_1}_{i^2},\underbrace{x_2,\dots, x_2}_{i},\underbrace{x_3,\dots, x_3}_{i},\dots),
x_j\in\{0_i,1_i\} \text{ for each $j\geq 1$}\}.$$
The map $\tilde{f}'$ is obtained by modifying $\tilde{f}$ in the obvious way.
\end{proof}

The final example shows the necessity of the ``pseudo-doubling'' in Theorem~\ref{thm:main} and
Corollary~\ref{cor:BDE PQ for pd}.

\begin{ejp}
\label{example:BDE not PQS}
There exists a compact, uniformly perfect, ultrametric space $X$ that is not pseudo-doubling and
a bounded distortion equivalence $f\co X\to X_3$ that is not $PQ$-symmetric.
Moreover, $f$ is not bi-H\"older.
\end{ejp}

\begin{proof}
For each $i\in\bn$ consider the edge $e_i$ on the tree $T_2\subseteq T_3$ joining $\bar 1_i$ to $\bar 1_{i+1}$.
Choose points  $\bar 1_i=x_i^0,x_i^1,\dots,x_i^i, x_i^{i+1}=\bar 1_{i+1}$
dividing $e_i$ into $i+1$ subintervals of length $1/(i+1)$.
For each $i\in\bn$ and to each point $x_i^j$, $1\leq j\leq i$, attach a copy $T_i^j$
of $T_2\subseteq T_3$ to $T_3$ by identifying the root $v_i^j$ of $T_i^j$ to $x_i^j$.
Let $T= T_3\bigvee_{i=1}^\infty\bigvee_{j=1}^i T_i^j$ be the resulting tree and let $X=end(T,v)$.
Define $f\co X\to X_3$ as follows. If $x\in (X\setminus end(T_2,v))\cup 0X_2$, define $f(x)=x$.
If $x\in 1X_2$ is of the form $x=\bar 1_i$ for some $i\in\bn$, define $f(x) = \bar 1_{i^2}y$.
Any other point $x\in X$ may be written (using self-explanatory notation) as
$x=\bar 1_ix_i^1\dots x_i^jz$, where $i\in\bn$, $1\leq j\leq i$, and 
$z=z_1z_2z_3\dots\in X_2$. For such an $x$, define $f(x) = \bar 1_{i+2j+z_1-1}z_2z_3z_4\dots$.
It can be verified that $f$ is a bounded distortion equivalence (with constant $\e$) and that $f$
is neither PQ-symmetric nor bi-H\"older.
\end{proof}


\bibliographystyle{plain}
\bibliography{Bounded_distortion}

\end{document}